\documentclass[a4paper, 11pt]{amsart} 

\usepackage{xcolor}
\usepackage{amsmath}
\usepackage{amssymb}			
\usepackage{amsfonts}
\usepackage{amsthm}
\usepackage{mathtools}
\usepackage{enumitem}
\usepackage{mathabx,epsfig}
\def\acts{\mathrel{\reflectbox{$\righttoleftarrow$}}}

\usepackage[colorlinks=true, urlcolor=blue, linkcolor=red]{hyperref}
\usepackage{tikz-cd}

\usepackage{tikz}
  \usetikzlibrary{matrix}
  \usetikzlibrary{fit}
  \usetikzlibrary{backgrounds}
  \usetikzlibrary{arrows}
  \usetikzlibrary{shapes}

\def\QQ{{\mathbb Q}}

\def\CC{{\mathbb C}}

\def\ZZ{{\mathbb Z}}
\def\FF{{\mathbb F}}

\def\fp{{\mathfrak p}}
\def\cores{{\mathrm{cor}}}
\def\res{{\mathrm{res}}}

\title[Endomorphism Algebra of $\mathrm{GL}_2$ Abelian Varieties]{On the Endomorphism Algebra of Abelian Varieties Associated with Hilbert Modular Forms} 
\author{Alireza Shavali} 
\newtheorem{theo}{Theorem}
\newtheorem{prop}{Proposition}
\newtheorem{cor}{Corollary}
\newtheorem{lemma}{Lemma}

\begin{document} 
\maketitle

\begin{abstract} In this article, we will generalize an explicit formula proved by Quer for the Brauer class of the endomorphism algebra of abelian varieties associated to modular forms of weight 2 to the case of Hilbert modular forms of parallel weight 2, under the condition that the degree of the base field over $\QQ$ is an odd number.
\end{abstract}
\section{Introduction}
Let $K$ be a totally real number field such that $[K:\QQ]$ is odd 
and let $f$ be a non-CM (Hilbert) newform of parallel weight 2, level
$N$ where $N$ is an ideal of $O_K$ and finite central character $\epsilon$. It is well-known that in this case one can use Shimura curves to construct an abelian variety $A_f$ over $K$ associated with $f$. Let $T_{\fp}$ be the Hecke operator at $\fp$ and $a_{\fp}$ the eigenvalue of $T_{\fp}$ acting on $f$. Then $E=\QQ(\{a_\fp\}_\fp)$ is a number field called the Hecke field of $f$. The abelian variety $A_f/K$ is of dimension $d=[E:\QQ]$ and hence its $\ell$-adic Tate module (after tensoring with $\QQ$) $V_{\ell}$ is of dimension $2d$ over $\QQ_\ell$. One can define an $E$-structure on this Tate module by letting $a_{\fp}$ act via the Hecke operator at $\fp$. This turns $V_\ell$ into a rank 2 free module over $E \otimes_{\QQ} \QQ_\ell$ endowed with a continuous $G_K$-action. This is the Galois representation associated with $f$ which after a choice of basis can be written as:
$$
\rho_{f,\ell}:G_K \rightarrow \mathrm{Aut}_E(V_\ell) \simeq \mathrm{GL}_2 (E \otimes_{\QQ} \QQ_{\ell})
$$
It is easy to see that this Galois representation is unramified outside $\ell N$ and for any unramified prime ideal $\fp$,
the Eichler-Shimura relation implies that the characteristic polynomial of $\rho_{f,\ell}(\mathrm{Frob}_\fp)$ is equal to
$X^2 - a_{\fp}X + \epsilon(\fp)\mathrm{Nm}(\fp)$. In particular, $\{ \rho_{f,\ell}\}_\ell$ is a compatible family of Galois representations. 

In appendix B of \cite{nekovavr2012level}, Nekovář studies the image of the Galois representation associated with a Hilbert modular form (not necessarily of weight 2) and generalizes results of Ribet \cite{ribet1980twists} and Momose \cite{momose1981adic} to this case. He constructs a division algebra $D$ over a subfield $F$ of the Hecke field $E$ which describes the image up to $p$-adic openness. In the special case where one knows there is an abelian variety associated with $f$ (in particular $f$ is of parallel weight 2) $F$ is equal to the center of the algebra $X:=\mathrm{End}_{\overline{\QQ}}(A_f) \otimes_{\ZZ}\QQ$ \cite[B.4.11]{nekovavr2012level} and since $A_f$ is of $\mathrm{GL}_2$-type over $K$ and $f$ is non-CM, it is a Ribet-Pyle abelian variety, i.e.  $E\simeq\mathrm{End}_{K}(A_f) \otimes_{\ZZ}\QQ$ is a maximal subfield of the simple algebra $X$ \cite[Propposition 3.1]{guitart2012abelian}. Moreover, $D$ and $X$ have the same class in the  Brauer group of $F$ \cite[B.4.11]{nekovavr2012level} and $D$ is the Mumford-Tate group of $A_f$ . 
It is natural to ask, if one can find explicit formulas for this class in the Brauer group in terms of the Hecke eigenvalues of $f$. When $K=\QQ$, Quer was able to prove such a formula \cite{quer1998classe}. This was later generalized to higher weights in \cite{ghate2005brauer} for the endomorphism ring of the motive associated with the form. 

Quer's result is about classical modular form. In this paper, we will generalize his result to Hilbert modular forms of parallel weight 2 over any odd degree extension $K$ of $\QQ$ assuming that the central character $\epsilon$ is trivial. 
In section \ref{RibGen} we will generalize a theorem of Ribet \cite[Theorem 5.5]{ribet1992abelian} to our situation. This is the main arithmetic input in the proof of Quer's formula. Ribet's proof works without much change but we will repeat the arguments for the convenient of the reader and because this does not seem to be written down in the literature in this case.  In section \ref{twistAlg} we will generalize \cite[Theorem 5.6]{ribet1992abelian} using the work of Chi \cite{chi1987twists}. Here some of the Galois cohomology computations become more complicated due to the fact that our base field $K$ is not contained in the field $F$, whereas the case of classical modular forms. Therefore one needs to carefully go up and down between some base fields to be able to carry out the computations. Finally, in section \ref{BrClass} we are able to prove Quer's formula in our setting.

\textbf{Acknowledgments.} I would like to thank Gebhard Böckle, Andrea Conti and Judith Ludwig for many valuable discussions about this work. This research was supported by the Deutsche Forschungsgemeinschaft (DFG)
through the Collaborative Research Centre TRR 326 Geometry and Arithmetic of Uniformized Structures, project number 444845124.

\section{Endomorphism Ring and Galois representation} \label{RibGen}

The main goal of this section is to generalize \cite[Theorem 5.5.]{ribet1992abelian} to the case of Hilbert modular forms. Ribet uses Faltings' theorem on isogenies (Tate conjecture) to relate the endomorphism algebra $X$ to the Tate module. We will do the same thing and closely follow Ribet's arguments. We will keep the assumptions and the notations from the first paragraph of the introduction.

Choose a prime number $\ell$ that splits completely in $E$. Then one has $d$ different
embeddings $\sigma:E \rightarrow \QQ_{\ell}$.
Let $M$ be a finite Galois extension of $K$
such that all of the endomorphisms of
$A_f$ are defined over $M$. Now by Faltings' isogeny theorem 
one has
\begin{equation}\label{Faltings}
X \otimes_{\QQ}\QQ_{\ell}=
\mathrm{End}_{\QQ_{\ell}[G_M]}(V_{\ell})
\end{equation}

Remember that $V_\ell$ also carries an $E$-structure through the Hecke action. 
Every embedding $\sigma$ of $E$ into
$\QQ_{\ell}$ gives a $E \otimes \QQ_{\ell}$
-module structure on $\QQ_{\ell}$ with
respect to which we can define 
$$
V_{\sigma} = V_{\ell} \otimes_{E\otimes \QQ_{\ell},\sigma}\QQ_{\ell}
$$
which is a $\QQ_{\ell}$-subspace of $V_{\ell}$
of dimension 2 that is invariant under
the action of $G_K$. Now note that
$a \in E$ acts on $V_{\sigma}$ via 
multiplication by $\sigma(a) \in E$ hence
for two different embeddings $\sigma$ and
$\tau$, $V_{\sigma}$ and $V_{\tau}$ have
trivial intersection as subspaces of $V_{\ell}$. This (together with obvious 
dimension reason) gives a decomposition
$$
V_{\ell} = \bigoplus_{\sigma:E\hookrightarrow\QQ_\ell} V_{\sigma} 
$$
of $\QQ_{\ell}[G_K]$-modules. 
The following lemma will be useful later.

\begin{lemma}
    For each embedding $\sigma$
    one has
    $\mathrm{End}_{\QQ_{\ell}[G_M]}(V_{\sigma}) = \QQ_{\ell}$. In particular, $V_{\sigma}$ is
    absolutely
    irreducible as a $G_M$-representation.
\end{lemma}
\begin{proof}
From (\ref{Faltings}) one has
$
X \otimes_{\QQ}\QQ_{\ell}=
\mathrm{End}_{\QQ_{\ell}[G_M]}(V_{\ell})
$.
Since $E$ is a maximal subfield of $X$,
taking the centralizer of $E \otimes
\QQ_{\ell}$
of both sides one gets
$$
E \otimes_{\QQ}\QQ_{\ell}=
\mathrm{End}_{E\otimes \QQ_{\ell}[G_M]}(V_{\ell})
$$
which means
$$
\oplus_{\sigma}\QQ_{\ell}= \oplus_{\sigma}
\mathrm{End}_{\QQ_{\ell}[G_M]}(V_{\sigma})
$$
which implies the first part. Since $V_{\sigma}$
is semi-simple by Faltings' proof of the Tate conjecture, irreducibility follows immediately. 
\end{proof}

For every prime $\mathfrak{p}$ of $O_K$ not
dividing $\ell N$, recall that $\mathrm{Frob}_{\mathfrak{p}}$ action on
$V_{\ell}$ has characteristic polynomial 
$$
X^2-a_{\mathfrak{p}}X+\epsilon(\mathfrak{p})N(\mathfrak{p}) \in E[X]
$$
therefore, for every embedding $\sigma:E\rightarrow \QQ_{\ell}$
one has
$$
\mathrm{tr}(\mathrm{Frob}_{\mathfrak{p}}
\acts V_{\sigma}) = \sigma (a_{\mathfrak{p}})
\in \QQ_{\ell}
$$
Restricting the compatible family of 
Galois representation to 
$G_M$,
one gets another compatible family, namely
for every finite place $v$ of $M$ not 
dividing $\ell N$ there is
$t_v \in E$ such that 
$$
\mathrm{tr}(\mathrm{Frob}_{v}
\acts V_{\sigma}) = \sigma (t_v)
\in \QQ_{\ell}
$$
Let $\Sigma_{\ell N}$ be the set of finite
places of $M$ not dividing $\ell N$ and $L=\QQ (t_v:v\in \Sigma_{\ell N}) \subset E$.
Then one has the following:
\begin{lemma}\label{center_is_L}
    The center of the algebra
    $\mathrm{End}_{\QQ_{\ell}[G_M]}(V_{\ell})$
    is $L \otimes_{\QQ} \QQ_{\ell}$.
\end{lemma}
\begin{proof}
    First note that by Faltings' theorem
    $$
    E \otimes \QQ_{\ell}=
    \mathrm{End}_{\QQ_{\ell}[G_K]}(V_{\ell})\subset
    \mathrm{End}_{\QQ_{\ell}[G_M]}(V_{\ell})
    $$
    and since $E \otimes \QQ_{\ell}$
    centralizes itself, it should contain
    the center of
    $\mathrm{End}_{\QQ_{\ell}[G_M]}(V_{\ell})$.

    By Semi-simplicity of $V_{\ell}$, one can
    see that $V_{\sigma}$ and $V_{\tau}$ are 
    isomorphic as $G_M$-representations if and only
    if they have the same $\mathrm{Frob}_v$ traces
    for all places $v$ of $M$ not dividing $\ell N$
    or equivalently $\sigma$ and $\tau$ agree
    on $L$. Now let
    $\gamma:L \rightarrow \QQ_{\ell}$ be an 
    embedding and define
    $$
    V_{\gamma}=\oplus_{\sigma|_{L}=\gamma}V_{\sigma}
    $$
    So one has the  decomposition $V=\oplus V_{\gamma}$
    and also since there is clearly no non-trivial
    endomorphism from one $V_{\gamma}$
    to another one also has the decomposition
    $$
    \mathrm{End}_{\QQ_\ell[G_M]}(V_{\ell})
    =\oplus_{\gamma}
    \mathrm{End}_{\QQ_\ell[G_M]}(V_{\gamma})
    $$
    Now let $a \in L$. Then $a$ acts on $V_{\sigma}$
    by $\sigma (a)$ hence it acts on the whole
    subspace $V_{\gamma}$ by the scalar $\gamma(a) \in \QQ_{\ell}$
    which means (because of the decomposition above)
    it's in the center of
    $\mathrm{End}_{\QQ_\ell[G_M]}(V_{\ell})$.
    So the $E$-algebra structure on 
    $\mathrm{End}_{\QQ_\ell[G_K]}(V_{\ell})$
    induces this $L$-algebra structure on 
    $Z(\mathrm{End}_{\QQ_\ell[G_M]}(V_{\ell}))$
    which means it's enough to prove
    $$
    Z(\mathrm{End}_{\QQ_\ell[G_M]}(V_{\ell}))
    \simeq
    L \otimes \QQ_{\ell}
    $$
    as $L$-algebras. This is easy to check:
    $$
    Z(\mathrm{End}_{\QQ_\ell[G_M]}(V_{\ell}))
    =Z(\oplus_{\gamma}
    \mathrm{End}_{\QQ_\ell[G_M]}(V_{\gamma}))
    =\oplus_{\gamma}Z(
    \mathrm{End}_{\QQ_\ell[G_M]}(V_{\gamma}))
    \simeq
    \oplus_{\gamma}\QQ_{\ell}
    =
    L \otimes \QQ_{\ell}
    $$
    
\end{proof}
\begin{cor}
    $L$ is the center of X, i.e. $L=F$.
\end{cor}
\begin{proof}
    Recall that from Faltings' isogeny theorem we had
    $$
    X \otimes_{\QQ}\QQ_{\ell}=
\mathrm{End}_{\QQ_{\ell}[G_M]}(V_{\ell})
    $$ Now from the last lemma:
    $$
    L \otimes \QQ_{\ell}=
    Z(\mathrm{End}_{\QQ_{\ell}[G_M]}(V_{\ell}))
    = Z(X \otimes_{\QQ}\QQ_{\ell})
    = F \otimes \QQ_{\ell}
    $$
    which implies $F=L$.
\end{proof}
\begin{lemma}\label{simpleArg}
    If $\sigma, \tau:E \rightarrow \QQ_{\ell}$
    are embeddings that agree on $F$ then there
    exists a character $\phi:G_K \rightarrow \QQ_{\ell}^{\times}$
    such that $V_{\sigma}\simeq V_{\tau} \otimes \phi$
    as representation of $G_K$.
\end{lemma}
\begin{proof}
    From the proof of lemma \ref{center_is_L} 
    we know that since $\sigma$ and $\tau$ 
    agree on $F=L$, $V_{\sigma}$ and $V_{\tau}$
    are isomorphic as representations of $G_M$. 
    So we can choose two bases for $V_{\sigma}$
    and $V_{\tau}$ such that the homomorphisms
    $\rho_{\sigma}:G_K \rightarrow GL_2(\QQ_{\ell})$
    and
    $\rho_{\tau}:G_K \rightarrow GL_2(\QQ_{\ell})$
    associated with $V_{\sigma}$ and $V_{\tau}$
    are equal on $G_M$. Now define
    $$
    \phi (g) := \rho_{\sigma}^{-1}(g)\rho_{\tau}(g)
    $$
    A priori $\phi$ is just a map
    $\phi:G_K \rightarrow GL_2(\QQ_{\ell})$ which
    is trivial on $G_M$. We want to prove that it
    is actually a homomorphism with values
    in the center (hence actually a character). 

    Let $g\in G_K$ and $h \in G_M$. 
    Note that
    $\rho_{\sigma} (h) = \rho_{\tau} (h)$
    and 
     $\rho_{\sigma} (ghg^{-1}) = \rho_{\tau} (ghg^{-1})$
     since $G_M$ is normal in $G_K$.
     Now the following computation shows that 
     $\phi (g)=\rho_{\sigma}^{-1}(g)\rho_{\tau}(g)$
     commutes with $\rho_{\tau} (h)$:
     $$
     \rho_{\sigma}^{-1}(g)\rho_{\tau}(g)\rho_{\tau} (h)
     = \rho_{\sigma}^{-1}(g)\rho_{\tau}(gh) =
     \rho_{\sigma}^{-1}(g)
     \rho_{\tau}(ghg^{-1})\rho_{\tau}(g)
     $$
     $$
     = \rho_{\sigma}(g^{-1})
     \rho_{\sigma}(ghg^{-1})\rho_{\tau}(g)
     = \rho_{\sigma} (h)\rho_{\sigma}^{-1}(g)\rho_{\tau}(g)
     = \rho_{\tau} (h)\rho_{\sigma}^{-1}(g)\rho_{\tau}(g)
     $$
     Now since 
     $\mathrm{End}_{\QQ_{\ell}[G_M]}(V_{\tau})=\QQ_{\ell}$
    we are done. 
\end{proof}
\begin{cor}
    Using the notation of the last lemma, $\phi^2=\frac{{}^\sigma \epsilon}{{}^\tau \epsilon}$
    and for any prime $\mathfrak{p}$ of $K$ of
    good reduction for $A_f$, one has
    $$
    {\sigma (a_{\mathfrak{p}})}=
    \phi(\mathrm{Frob}_{\mathfrak{p}})
    \tau (a_{\mathfrak{p}})
    $$
\end{cor}
\begin{proof}
    Note that 
    $V_{\sigma}\simeq V_{\tau} \otimes \phi$.
    Taking determinants of both sides one gets
    the first part and taking trace one gets the 
    second part.
\end{proof}

Ribet also proves that in the $K=\QQ$ case, the field $F$ is generated by $\{a_p^2/\epsilon(p)\}_{p\nmid N}$. This is also true in our case. In fact, by \cite[B.4.11]{nekovavr2012level} $F$ is exactly the field fixed by inner-twists and the above result is known in much more generality in this context by the results of \cite{conti2023big}.

\begin{prop}[\cite{conti2023big}, Corollary 4.12] \label{CLM}
    The field $F$ is generated over $\QQ$ by 
    numbers $a_\fp^2/\epsilon(\fp)$ for $\fp \nmid N$.
\end{prop}

Now let $\sigma\in G_K$, then $\sigma$
acts on the $\overline{\QQ}$-endomorphisms of $A_f$ by 
$\sigma(\phi)(x):=\sigma(\phi(\sigma^{-1}x))$ and this linearly extends to an action on $X$. $E$ is clearly invariant under the action of $G_K$ on X (we are identifying $E$ with the maximal subfield of $X$). Since this
is an automorphism of $F$-algebras, By the Skolem-Noether theorem
the action of $\sigma$ is given by conjugation by some
element $\alpha(\sigma)\in X$. Since $E$ is invariant under the
Galois action, $\alpha(\sigma)$ commutes with $E$ and therefore
$\alpha(\sigma)\in E$ because $E$ is a maximal subfield and hence its own centralizer. The next theorem relates the map $\alpha$ which is of geometric (motivic) nature to the (automorphic) data of Hecke eigenvalues. 

\begin{theo}\label{alpha2trivial}
    For every $\sigma \in G_K$ one has $\alpha (\sigma)^2 / \epsilon(\sigma) \in F^{\times}$.
    Moreover, for every prime ideal $\mathfrak{p}$
    of $O_K$ away from $\ell N$, if $a_{\mathfrak{p}}\neq 0$ then
    $\alpha (\mathrm{Frob}_{\mathfrak{p}}) \equiv
    a_{\mathfrak{p}}$ modulo $F^{\times}$.
\end{theo}
\begin{proof}
    As usual, let $\ell$ be a prime number that
    splits completely in $E$.
    It enough to prove that for every pair of embeddings
    $\sigma$ and $\tau$ of $E$ in $\QQ_{\ell}$
    that agree on $F$ one has 
    $\sigma(\alpha^2 /\epsilon)=\tau(\alpha^2 /\epsilon)$.

    Now if $\sigma$ and $\tau$ agree on $F$ then
    by lemma \ref{simpleArg} there exist
    a character $\phi:G_K \rightarrow \QQ_{\ell}^{\times}$
    such that $V_{\sigma}\simeq V_{\tau} \otimes \phi$
    as representation of $G_K$. 
    In particular, it implies that as 1-dimensional
    representation of $G_K$ one has
    $$
    \mathrm{Hom}_{\QQ_{\ell}[G_M]}(V_{\sigma}
    ,V_{\tau}) \simeq \phi
    $$
    Also note that if $\sigma$ and $\tau$ don't
    agree on $F$ then they are not isomorphic
    as $G_M$-representation hence
    $$
    \mathrm{Hom}_{\QQ_{\ell}[G_M]}(V_{\sigma}
    ,V_{\tau}) = 0
    $$
    Therefore we can completely understand
    $X \otimes \QQ_{\ell}$:
    $$
    \mathrm{End}^{0}_M (A_f) \otimes \QQ_{\ell}
    \simeq 
    \mathrm{End}_{\QQ_{\ell}[G_M]}(\oplus V_{\gamma})
    =\oplus_{\sigma,\tau}
    \mathrm{Hom}_{\QQ_{\ell}[G_M]}(V_{\sigma}
    ,V_{\tau})
    $$
    Now remember that on the LHS, $g\in G_K$
    acts via conjugation by $\alpha(g)$. Hence,
    it acts on $V_\sigma$ and $V_\tau$ by
    $\sigma(\alpha(g))$ and $\tau(\alpha(g))$
    respectively. Now assume that $\sigma$ and
    $\tau$ agree on $F$. Then $g$ acts on
    $\mathrm{Hom}_{\QQ_{\ell}[G_M]}(V_{\sigma}
    ,V_{\tau})$
    by $\sigma(\alpha(g))/\tau(\alpha(g))$. On the
    other hand as a representation of $G_K$ this
    is just $\phi$ so
    $\sigma(\alpha(g))/\tau(\alpha(g)) = \phi (g)$. 
    Since $\phi^2=\frac{{}^\sigma \epsilon}{{}^\tau \epsilon}$ one deduces that $\sigma(\alpha^2 /\epsilon)=\tau(\alpha^2 /\epsilon)$ and the result follows.

    For the second part, first notice that
    $$
    \phi (\mathrm{Frob}_{\mathfrak{p}})
    =
    \sigma(\alpha(\mathrm{Frob}_{\frak{p}}))/
    \tau(\alpha(\mathrm{Frob}_{\frak{p}}))
    =
    \sigma(a_{\frak{p}})/
    \tau(a_{\frak{p}})
    $$
    therefore 
    $$
    \sigma(\alpha(\mathrm{Frob}_{\frak{p}})
    /a_{\frak{p}})
    =
    \tau(\alpha(\mathrm{Frob}_{\frak{p}})
    /a_{\frak{p}})
    $$
    which implies the result.
\end{proof}

\section{Twisted Algebras and Galois Cohomology} \label{twistAlg}

The goal of this section is to prove an analogue of \cite[Theorem 5.6]{ribet1992abelian} in our setting. Ribet uses a result of Chi to prove this theorem.
In \cite{chi1987twists}, Chi studies the twists of a central simple algebra by a 1-cocycle. We need to review some of his results and generalize some of those to our setting. 

First note that the endomorphism ring $\mathrm{End}_{\overline{\QQ}}(A_f)$ acts
on the space of differential 1-forms on
$A_f/\overline{\QQ}$ (which we denote by $\Omega^{1}_{\overline{\QQ}}$)  via pull back. For an
endomorphism $\phi$ and a 1-form $\omega$ we use the usual
notation $\phi^{*}\omega$ for this action. 
This action linearly extends
to an action of $X$ on this space and we use the same notation
for this action as well. Also, note that for 
any $\sigma \in G_K$ and $\phi \in X$ one has
\begin{equation} \label{X_action}
    \sigma (\phi^{*}\omega)=(\sigma \phi)^{*}(\sigma\omega)
    = (\alpha(\sigma)\cdot \phi \cdot \alpha(\sigma)^{-1})^{*}
    (\sigma\omega)
\end{equation}

For $\sigma$ and $\tau$ in $G_K$, define $c_{\alpha}(\sigma,\tau)
:=\alpha(\sigma)\alpha(\tau)\alpha(\sigma\tau)^{-1}$. This commutes
with every element in $X$
so it lands in $F$. Therefore $\alpha$
gives a well-defined group homomorphism
$$
\alpha_K: G_K \rightarrow \frac{E^{\times}}{F^{\times}}
$$
Let $\alpha_{FK}$ be the
restriction of $\alpha_K$ to $G_{FK}$. We sometimes use the same
notation to denote the composition of this map with the canonical 
map to $(EK)^{\times}/(FK)^{\times}$:
$$
\alpha_{FK}: G_{FK} \rightarrow \frac{E^{\times}}{F^{\times}}
\rightarrow \frac{(EK)^{\times}}{(FK)^{\times}}
$$
Let $X_{FK}:=X \otimes_{F} FK$. This is an algebra
over $FK$. Note that every element in $FK$ is a sum of
the form $\sum_i f_ik_i$ for $f_i \in F$ and $k_i \in K$
so $X_{FK}$ is generated by pure tensors of the form
$\sum_i \phi_i \otimes k_i$ for $k_i \in K$. 

As in \cite{chi1987twists} one can look at the twist of this
algebra with (the 1-cocycle defined by) $\alpha$
which we denote by $X_{FK}(\alpha_{FK})$ following Chi. 

\begin{prop}[\cite{chi1987twists}, Proposition 1.1]
    One has 
    $$
    \dim_{FK} X_{FK}(\alpha_{FK}) = \dim_{FK} X_{FK} = dim_{F} X
    $$
    and Moreover
    $$
    X_{FK}(\alpha_{FK}) \otimes_{FK} \overline{\QQ} \simeq 
    X_{FK} \otimes_{FK} \overline{\QQ}
    $$
    Therefore, $X_{FK}(\alpha_{FK})$ is a central simple $FK$-algebra.
\end{prop}

One can also view $X_{FK}$ as a $K$-algebra and twist it with
$\alpha_K$ instead to get the $K$-algebra $X_{FK}(\alpha_K)$.
Let us recall the definition of this algebra. First for any
$\sigma \in G_K$ we define the twisted action of $\sigma$ on
$X_{FK}\otimes_K \overline{\QQ}$ as follow. On pure tensors
of the form $\phi \otimes k \otimes \lambda$ for $\phi \in X$,
$k \in K$ and $\lambda \in \overline{\QQ}$ we define:
$$
tw(\sigma)\cdot(\phi \otimes k \otimes \lambda):=\alpha(\sigma)\phi
\alpha(\sigma)^{-1}\otimes k \otimes \sigma(\lambda)
$$
Note that $k=\sigma(k)$ in the above expression.  Now we define:
$$
X_{FK}(\alpha_K):=(X_{FK}\otimes_{K}\overline{\QQ})^{tw(G_K)}
$$
This $K$-algebra also has the structure of an $FK$-algebra via
$a\cdot \sum\psi_i\otimes\lambda_i := \sum a\psi_i\otimes \lambda_i$ for $a\in FK$, $\psi_i \in X_{FK}$ and $\lambda_i \in \overline{\QQ}$.

\begin{prop}[\cite{chi1987twists}, Proposition 1.2]
    One has $X_{FK}(\alpha_{FK})\simeq X_{FK}(\alpha_{K})$
    as $FK$-algebras.
\end{prop}

This implies that 
$X_{FK}(\alpha_{K})=(X_{FK}\otimes_{K}\overline{\QQ})^{tw(G_K)}$
is also a central simple $FK$-algebra. From now on we simply
write $X_{FK}(\alpha)$ for this central simple algebra.

$E$ is a subfield of $X_{FK}$. Let $L$ be a maximal subfield
of $X_{FK}$ containing $E$. Then $L$ contains $EK$ as well. 
So one can look at $\alpha_{FK}$ as a group homomorphism
$$
\alpha_{FK}:G_{FK} \rightarrow \frac{L^{\times}}{(FK)^{\times}}
$$
which has values in $E$. Now one can apply \cite[Proposition 2.4]{chi1987twists} to get:
$$
X_{FK}(\alpha_{FK}) \otimes_{FK} \mathrm{End}_{FK} L \simeq
X_{FK} \otimes_{FK} \mathrm{End}_{FK} L (\alpha_{FK})
$$
So in the $Br(FK)$ one has
$$
[X_{FK}(\alpha)] = [X_{FK}] + [\mathrm{End}_{FK} L (\alpha_{FK})]
$$
From this point onward, we assume that the central character $\epsilon$ of $f$ is trivial. In the general case, one also needs to carry the 2-cocycle $c_\epsilon = [\mathrm{End}_{FK} L (\epsilon)]$ in the calculations as in \cite{quer1998classe} which complicates some of the computations. Having this assumption, now we can prove:

\begin{lemma}\label{Schur_index} Assuming $\epsilon$ is trivial,
    the order of $[X_{FK}(\alpha)]$ in $Br(FK)$ divides 2. 
\end{lemma}
\begin{proof}
So far we proved 
$$
[X_{FK}(\alpha)] = [X_{FK}] + [\mathrm{End}_{FK} L (\alpha_{FK})]
$$
in $\mathrm{Br}(FK)$.
    By \cite[Proposition B.4.12]{nekovavr2012level} we know that $X$ and hence $X_{FK}$ 
have Schur index dividing 2. Also from theorem
\ref{alpha2trivial}
we know
that $\alpha^2=\epsilon$ modulo $F^\times$. Applying \cite[Proposition 2.2]{chi1987twists} we get:
$$
2\cdot [\mathrm{End}_{FK} L (\alpha_{FK})] = 
[\mathrm{End}_{FK} L (\alpha_{FK}^2)]
= [\mathrm{End}_{FK} L (\epsilon)] 
$$
Since $\epsilon=1$ we are done.
\end{proof}
From section 2 of \cite{chi1987twists} we know  that $[\mathrm{End}_{FK}L(\alpha_{FK})]$
in $Br(FK)=H^2(G_{FK}, \overline{\QQ})$ is the same as the image
of the cohomology class defined by $\alpha$ in 
$H^1(G_{FK},PGL_n(\overline{\QQ}))$ under the connecting
homomorphism
$$
\delta : H^1(G_{FK},PGL_n(\overline{\QQ})) \rightarrow
H^2(G_{FK}, \overline{\QQ})
$$
where $n=[L:FK]$. More concretely, one can view every 
$\ell \in L$ as an $FK$-linear endomorphism $\ell:L\rightarrow L$
given by multiplication by $\ell$. So every $\ell$ can be
viewed as an $n\times n$ matrix with $FK$-entries. Now viewing
every $\alpha(\sigma) \in E$ as such a matrix, conjugation by
this matrix gives an element in $PGL_n(FK)\subset PGL_n(\overline{\QQ})$.
This gives 1-cocycle with $PGL_n(FK)$ or rather 
with $PGL_n(\overline{\QQ})$ values that is invariant under
the $G_{FK}$ action. Since the connecting homomorphism
$\delta$ sends a 1-cocycle $f$ to $f(\sigma)\sigma(f(\tau))f(\sigma, \tau)^{-1}$ one concludes:
\begin{cor}
    Let 
    $c_\alpha (\sigma,\tau)=
    \alpha(\sigma)\alpha(\tau)\alpha(\sigma\tau)^{-1}$ 
    be a 2-cocycle for the trivial action of $G_K$ on
    $F^{\times}$. Then the image of $[c_\alpha]$ under the 
    sequence
    $$
    H^2(G_K,F^{\times}) \xrightarrow{\res} 
    H^2(G_{FK},F^{\times}) \xrightarrow{\iota_{*}}
    H^2(G_{FK},\overline{\QQ}^{\times})
    $$
    is exactely the class of
    $[X_{FK}(\alpha)]$ in $H^2(G_{FK}, \overline{\QQ})
    =Br(FK)$
\end{cor}
\begin{cor}\label{cocycle}
    In $Br(FK)$ one has:
    $$
    [X_{FK}(\alpha)] = [X_{FK}] + \iota_{*}(\res([c_\alpha]))
    $$
\end{cor}
Our next goal is to prove that $X_{FK}(\alpha)$ is trivial
in the Brauer group. The main ingredient is the next proposition.
\begin{prop}
    $X_{FK}(\alpha)$ acts (linearly) on $\Omega^1_{K}$.
\end{prop}
\begin{proof}
    First, we defined an action of $X_{FK}\otimes \overline{\QQ}$ on 
    $\Omega^1_{\overline{\QQ}}$ by extending the action of
    $X$ linearly, namely we define:
    $$
    (\phi \otimes k \otimes \lambda)^{*}\omega :=
    k\lambda \phi^{*}\omega
    $$
    for $\phi \in X$, $k \in K$ and $\lambda \in \overline{\QQ}$. 
    Now using 
    (\ref{X_action}) one easily sees that for any
    $\sigma \in G_K$ and $\psi \in X_{FK}\otimes \overline{\QQ}$:
    $$
    \sigma(\psi^{*}\omega) = (tw(\sigma)\cdot \psi)^{*}\sigma\omega
    $$
    which means that if $\psi$ is invariant under the twisted
    Galois action and $\omega$ is invariant under the usual
    Galois action, then $\psi^{*}\omega$ is also invariant.
    This means that $X_{FK}(\alpha)$ acts on $\Omega^1_{K}$.
\end{proof}
\begin{prop}\label{twist_trivial}
    $X_{FK}(\alpha) \in Br(FK)$ is trivial.
\end{prop}
\begin{proof}
    Let $X_{FK}(\alpha)=M_n(D)$ for some division algebra $D$
    over $FK$ of dimension $s^2$. By corollary \ref{Schur_index}
    one has $s|2$. Now $\dim X_{FK}(\alpha)=n^2s^2$ which should
    be equal to the dimension of $X$ over $F$ therefore 
    $ns=[E:F]$. By the last proposition $\Omega^1_K$ is a
    $M_n(D)$-module. So there is a $D$-vector space $W$
    such that $\Omega^1_K \simeq W^n$. The dimension of
    $\Omega^1_K$ over $K$ is equal to the dimension of the
    abelian variety $A_f$ which is $[E:\QQ]$. Hence
    $$
    s^2=\dim_{FK}D|\dim_{FK} W = \frac{[E:\QQ]}{n[FK:K]}
    =\frac{ns[F:\QQ]}{n[F:F\cap K]}=s[F\cap K:\QQ]
    $$
    This implies $s|[F\cap K:\QQ]$ but since $s|2$ and 
    $[K:\QQ]$ is odd, one has $s=1$.
\end{proof}
From proposition \ref{twist_trivial} and corollary \ref{cocycle}
and the fact that $[X_{FK}]\in Br(FK)$ has order dividing 2,
one deduces:
\begin{cor}
    In $Br(FK)$ one has
    $$
    [X_{FK}] = \iota_{*}(\res([c_\alpha]))
    $$
\end{cor}
Now we need to go down from $FK$ to $F$ to compute
the class $[X]$ in $Br(F)$ using $\alpha$. 
We can use the corestriction map to do so. First note that by
the last corollary we know that following the below diagram,
the image of
$[c_\alpha]$ in $H^2(G_{FK},\overline{F}^*)$
is $[X_{FK}]$ which is the image of $[X]$ under the
restriction. 

\begin{center}
\begin{tikzcd}
{[c_\alpha]\in H^2(G_K,F^*)} \arrow[rr, "\res"]      &  & {H^2(G_{FK},F^*)} \arrow[dd, "\iota_{*}"]                             \\
                                                    &  &                                                                       \\
{[X]\in H^2(G_F,\overline{F}^*)} \arrow[rr, "\res"'] &  & {  H^2(G_{FK},\overline{F}^*)} \arrow[ll, "\cores"', dotted, bend right]
\end{tikzcd}
\end{center}
This means that 
$$
\iota_{*}(\res([c_\alpha])) = \res([X])
$$
On the other hand, $\cores \circ \res = [FK:F]=[K:F\cap K]$
which is an odd integer.
Since $X$ has order dividing 2 in the Brauer group,
$\cores(\res([X]))=X$. 

Finally, we can conclude the generalization of \cite[Theorem 5.6]{ribet1992abelian} to the case of Hilbert modular form (with trivial central character): 
\begin{cor}\label{corRes}
    In $Br(F)$ one has
    $$
    [X]=\cores(\iota_{*}(\res([c_\alpha])))
    $$
\end{cor}

\section{Computing the Brauer Class}\label{BrClass}

Now we have all the ingredients to generalize
\cite{quer1998classe}. The proof is essentially the same. First notice that from theorem 
\ref{alpha2trivial} (and the assumption $\epsilon=1$) we know that $\alpha^2$ is
trivial, hence
$$
\alpha^2 : G_K \rightarrow F^{\times}/(F^{\times})^2
$$
is a homomorphism. Let $N$ be the finite Galois
extension of $K$ associated with its kernel, i.e.
$\mathrm{ker}(\alpha^2)=G_N$. Since
$\mathrm{Gal}(N/K)\simeq \mathrm{Im}(\alpha^2)
\subset F^{\times}/(F^{\times})^2$
is a 2-torsion group, one has $\mathrm{Gal}(N/K)\simeq (\ZZ/2\ZZ)^m$ for some positive integer
$m$. Therefore, $N=K(\sqrt{t_1},...,\sqrt{t_m})$
for some $t_i \in K$
and if one defines $\sigma_i \in \mathrm{Gal}(N/K)$ 
with the relations
$$
\sigma_i (\sqrt{t_j}) = (-1)^{\delta_{i,j}}\sqrt{t_j}
$$
then
$\sigma_1,...,\sigma_m$ form an $\FF_2$-basis for
$\mathrm{Gal}(N/K)$. 
\begin{lemma}\label{HilbertSym}
    In $\mathrm{Br}(FK)$ on has:
    $$
    \iota_{*} (\mathrm{res}([c_{\alpha}]))
    = (t_1,\alpha(\sigma_1)^2)
    (t_2,\alpha(\sigma_2)^2)
    \cdots 
    (t_m,\alpha(\sigma_m)^2)
    $$
    where $(a,b)=(a,b)_{FK}$ denotes the Hilbert symbol.
\end{lemma}
\begin{proof}
    First notice that since
    $\alpha(\sigma)\sigma(\alpha(\tau))\alpha(\sigma\tau)^{-1}$
    is a coboundary, the 2-cocycle
    $[c_\alpha]$ is also given by the formula
    $(\sigma,\tau)\mapsto \frac{\alpha(\tau)}{\sigma(\alpha(\tau))}$.
    
    For each $\tau \in G_K$ let 
    $$\tau(\sqrt{t_i})=(-1)^{x_i(\tau)}\sqrt{t_i}$$
    Then $x_i:G_K \rightarrow \ZZ/2\ZZ$
    is clearly a group homomorphism. Similarly
    let $y_i:G_{FK}\rightarrow \ZZ/2\ZZ$
    be the homomorphism given by
    $$
    \sigma(\alpha(\sigma_i))=(-1)^{y_i(\sigma)}\alpha(\sigma_i)
    $$
    for $\sigma \in G_{FK}$.
    Now since $\{ \sigma_i \}_{i=1}^m$ provides 
    an $\FF_2$ basis for $\mathrm{Gal}(N/K)$,
    every element $\tau \in G_K$ can be
    written as $\eta \prod_{i=1}^m \sigma_i^{x_i(\tau)}$
    where $\eta$ is in $G_N=\mathrm{ker}(\alpha^2)$.
    Applying $\alpha^2$ to both sides one
    gets
    $$
    \alpha^2(\tau) \equiv \prod_{i=1}^m \alpha^2(\sigma_i)^{x_i(\tau)} \:\:
    \mathrm{(mod} \: F^{\times 2}\mathrm{)}
    $$
    which implies 
    $$
    \alpha(\tau) = \lambda\prod_{i=1}^m \alpha(\sigma_i)^{x_i(\tau)}
    $$
    For some $\lambda \in F^{\times}$.
    Now one can use this to give a description 
    of $[c_\alpha]$. Applying $\sigma \in G_{FK}$ to the
    both sides one has
    $$
    \sigma(\alpha(\tau))
    =\lambda \prod_{i=1}^m \sigma(\alpha(\sigma_i))^{x_i(\tau)}
    =\lambda\prod_{i=1}^m (-1)^{y_i(\sigma)x_i(\tau)}\alpha(\sigma_i)^{x_i(\tau)}
    =\alpha(\tau)\prod_{i=1}^m (-1)^{y_i(\sigma)x_i(\tau)}$$
    which gives the description
    $$
    \prod_{i=1}^m (-1)^{y_i(\sigma)x_i(\tau)}
    $$
    for $\iota_{*} (\mathrm{res}([c_{\alpha}]))$
    in $Br(FK)$.
    Now, 
    it is well-known
    that the 2-cocycle $(-1)^{y_i(\sigma)x_i(\tau)}$
    is represented by the
    Hilbert symbol $(t_i,\alpha^2(\sigma_i))$
    so we are done.
\end{proof}

From \cite{nekovavr2012level} we know that $\Gamma \simeq \mathrm{Gal}(E/F)$ is the group of inner-twists 
of the form $f$. Namely, for each $\sigma \in \mathrm{Gal}(E/F)$ there exist a unique
character $\chi_{\sigma}:G_K \rightarrow \CC^{\times}$ such that
$\chi_\sigma \otimes f = {}^\sigma f$. This is
equivalent to saying that for every finite
place $\fp$ of $K$ not dividing
$N$ one has 
$$\chi_{\sigma}(\mathrm{Frob}_\fp)\cdot a_\fp = \sigma(a_\fp)$$
where $a_\fp$ is the $\fp$'th Fourier coefficient
(Hecke eigenvalue) of $f$. 

\begin{lemma}
    The characters $\chi_{\sigma}$ appearing
    in the inner-twists are exactly characters
    of $G_K$ that factor through 
    $\mathrm{Gal}(N/K)$. In partcular, the
    number of the inner-twists of $f$ is $2^m$.
\end{lemma}
\begin{proof}
    First, we prove that all $\chi_{\sigma}$'s
    are trivial of $G_N=\mathrm{ker}(\alpha^2)$. 
    The Sato-Tate conjecture for Hilbert modular
    forms is known by \cite{barnet2011sato}. This implies that
    the set of prime ideal $\mathfrak{p}$
    of $O_K$ for which $a_{\mathfrak{p}} \neq 0$
    has density 1. Then by
    Chebotarev's density theorem 
    the Frobenius elements of these primes
    are dense in $G_K$, therefore it's
    enough to check that $\chi_{\sigma}$
    is trivial on the
    elements of the
    form $\mathrm{Frob}_{\mathfrak{p}} \in G_K$ that are
    in the kernel of $\alpha^2$ and $a_{\mathfrak{p}} \neq 0$. 
    
    Now if $a_{\mathfrak{p}} \neq 0$ then by
    theorem \ref{alpha2trivial},
    $\alpha^2(\mathrm{Frob}_{\mathfrak{p}})
    \equiv a_{\mathfrak{p}}^2$ modulo $F^{\times 2}$. Hence, if 
    $\mathrm{Frob}_{\mathfrak{p}} \in \mathrm{ker}(\alpha^2)$ then $a_{\mathfrak{p}} \in F$. This implies
    that $\chi_{\sigma} (\mathrm{Frob}_{\mathfrak{p}})=1$ by the
    definition of an inner-twist. So we
    are done.

    To prove that these are all such characters
    it's enough to prove that the number of
    character factoring through 
    $\mathrm{Gal}(N/K)$ is equal to the number
    of the inner-twists. 
    The group of
    character factoring through 
    $\mathrm{Gal}(N/K)$
    is the dual group of $\mathrm{Gal}(N/K)$
    and since this is abelian it has exactly
    $\frac{1}{2^m}$ elements. 
    Then by Chebotarev's density theorem the
    density of primes $\mathfrak{p}$ (with $a_{\mathfrak{p}} \neq 0$) that 
    $\mathrm{Frob}_{\mathfrak{p}} \in G_N$
    or equivalently $a_{\mathfrak{p}} \in F^{\times}$
    is $\frac{1}{2^m}$.

    Now, notice that if $(\sigma, \chi_{\sigma})$
    is an inner-twist then by definition
    $\chi_{\sigma}(\mathrm{Frob}_{\mathfrak{p}}) \cdot a_{\mathfrak{p}} = 
    \sigma(a_{\mathfrak{p}})$. So all
    $\chi_{\sigma}$'s are trivial on $\mathrm{Frob}_{\mathfrak{p}}$ if and
    only if $a_{\mathfrak{p}} \in F$. 
    Also, since $a_{\mathfrak{p}}^2 \in F$ for 
    all $\mathfrak{p}$, $\chi_{\sigma}^2 =1$.
    By \cite[Proposition B.3.3]{nekovavr2012level} $\Gamma$ is a finite 2-torsion
    abelian group. Hence,
    $\Gamma \simeq (\ZZ/ 2\ZZ)^n$ for some $n$.
    Clearly, $n \leq m$ since $\chi_{\sigma}$'s
    factor through $\mathrm{Gal}(N/K)$. 
    Now choose an $\FF_2$ basis $\sigma^{(1)},\cdots,\sigma^{(n)}$ for $\Gamma=\mathrm{Gal}(E/F)$. Let $G_M$ be the intersection
    of kernel of all $\chi_{\sigma}$'s
    which is equal to the intersection of the
    kernel of all $\chi_{\sigma^{(i)}}$'s.
    Now by Chebotarev's density theorem
    $M=N$ because they contain the same 
    Frobenius elements of $G_K$. Since $G_N$
    is the intersection of kernel of 
    $\chi_{\sigma^{(i)}}$'s which are all of order 2, one deduces that $n \geq m$.
    This implies $n=m$ and we are done.

\end{proof}

By the last lemma, the group of characters
$\chi_{\sigma}$ is the dual group of
$\mathrm{Gal}(N/K) \simeq (\ZZ/2\ZZ)^m$. 
Recall that $\{\sigma \}_{i=1}^n$ is an
$\FF_2$ basis for $\mathrm{Gal}(N/K)$
satisfying
$
\sigma_i (\sqrt{t_j}) = (-1)^{\delta_{i,j}}\sqrt{t_j}
$ where $N=K(\sqrt{t_1},...,\sqrt{t_m})$.
Now let $\sigma^{(1)},\cdots,\sigma^{(m)}$
be a dual basis for this (so each
$\sigma^{(i)}$ appear in an inner-twist), i.e.
$$
\sigma^{(j)}(\sigma_i)=(-1)^{\delta_{i,j}}.
$$

Notice that the fixed field of $\mathrm{ker}(\sigma^{(j)})$ is just $K(\sqrt{t_j})$.

Recall that we need to apply the corestriction map to get 
back over $F$ and find a formula for $[X]$ in
$\mathrm{Br}(F)$. The following well-known lemma
helps us
to do that.
\begin{lemma}[\cite{serre2013local}, Exercise XIV.2.4]
    Let $L/F$ be a finite separable extension
    and let $\mathrm{cor}: \mathrm{Br}(L)
    \rightarrow \mathrm{Br}(F)$ be the 
    corestriction map. Then for any
    $a\in L^{\times}$ and $b \in F^{\times}$
    one has
    $$
    \mathrm{cor}(a,b)_L = (N_{L/F}(a),b)_F
    $$
\end{lemma}
Now we can finally state and prove our main theorem. Note that for any finite place $\fp$ away from $N$ one has $a_\fp^2 \in F$ by proposition \ref{CLM}.
\begin{theo}
    Let $\mathfrak{p}_1,\cdots,\mathfrak{p}_m$
    be a set of prime ideals of $O_K$ not dividing $N$ and with 
    $a_{\mathfrak{p}_i} \neq 0$ such that
    $\sigma_i = \mathrm{Frob}_{\mathfrak{p}_i}$
    in $\mathrm{Gal}(N/K)$ (such primes
    exist by Chebotarev's theorem). Then
    In $\mathrm{Br}(F)$ one has:
    $$
    [X]
    = (N_{FK/F}(t_1),a_{\mathfrak{p}_1}^2)
    (N_{FK/F}(t_2),a_{\mathfrak{p}_2}^2)
    \cdots 
    (N_{FK/F}(t_m),a_{\mathfrak{p}_m}^2)
    $$
    where $(a,b)=(a,b)_{F}$ denotes the Hilbert symbol. 
\end{theo}
\begin{proof}
    Using lemma \ref{HilbertSym} one only
    needs to notice that
    $\alpha(\mathrm{Frob}_{\mathfrak{p}_i})^2
    \equiv a_{\mathfrak{p}_i}^2$ modulo
    $F^{\times 2}$, so they only differ by 
    a square which doesn't affect the Hilbert
    symbol. Therefore:
    $$
    \iota_{*} (\mathrm{res}([c_{\alpha}]))
    = (t_1,a_{\mathfrak{p}_1}^2)
    (t_2,a_{\mathfrak{p}_2}^2)
    \cdots 
    (t_m,a_{\mathfrak{p}_m}^2)
    $$
    Now one applies to corestriction map
    to both sides. 
    The left hand sides gives us
    $[X]$ by corollary \ref{corRes} and the
    right hand side give us 
    $$(N_{FK/F}(t_1),a_{\mathfrak{p}_1}^2)
    (N_{FK/F}(t_2),a_{\mathfrak{p}_2}^2)
    \cdots 
    (N_{FK/F}(t_m),a_{\mathfrak{p}_m}^2)$$
    by the previous lemma,
    since $a_{\mathfrak{p}_m}^2 \in F^{\times}$.
    This proves the statement of the theorem.
\end{proof}

\bibliographystyle{amsplain}
\bibliography{refs}

\end{document}